\documentclass{amsart}
\usepackage{amssymb,amsmath,amsthm,graphics,amscd,amsfonts}
\usepackage{color}

\theoremstyle{plain}
\newtheorem{theorem}{Theorem}[section]
\newtheorem{proposition}[theorem]{Proposition}
\newtheorem{corollary}[theorem]{Corollary}
\newtheorem{lemma}[theorem]{Lemma}

\newtheorem{definition}[theorem]{Definition}

\newtheorem{remark}[theorem]{Remark}

\newcommand{\barj}{{\overline j}}

\newcommand{\barpartial}{{\overline \partial}}

\newcommand{\mapright}[1]{\smash{\mathop{   \hbox to 0.7cm{\rightarrowfill}}
 \limits^{#1}}}

\def\ba{\begin{array}}
\def\ea{\end{array}}

\begin{document}

\title
{Lower diameter bounds for compact shrinking Ricci solitons}

\author{Akito Futaki}
\address{Department of Mathematics, Tokyo Institute of Technology, 2-12-1,
O-okayama, Meguro, Tokyo 152-8551, Japan}
\email{futaki@math.titech.ac.jp}

\author{Yuji Sano}
\address{Department of Mathematics, Kyushu University,
744 Motooka, Nishi-ku, Fukuoka-city, Fukuoka 819-0395 Japan}
\email{sano@math.kyushu-u.ac.jp}

\date{July 09, 2010}

\begin{abstract} 
It is shown that the diameter of a compact shrinking Ricci soliton has a universal lower bound.
This is proved by extending universal estimates for the first non-zero eigenvalue of Laplacian
on compact Riemannian manifolds with lower Ricci curvature bound 
to a twisted Laplacian on compact shrinking Ricci solitons.
\end{abstract}

\thanks{The first author is supported by MEXT, Grant-in-Aid for Scientific Research (A),
No. 21244003 and Challenging Exploratory Research No. 20654007. 
The second author is supported by MEXT, Grant-in-Aid for Young Scientists (B),
No. 22740041.}

\keywords{shrinking Ricci soliton, diameter bound}

\subjclass{Primary 53C21, Secondary 53C20}

\maketitle

\section{Introduction}

In this paper we show that the diameter of compact shrinking Ricci solitons have a universal lower bound.
Recall that Ricci solitons were introduced by Hamilton in \cite{Ham88} and are self-similar solutions to the Ricci flow.
They are defined as follows.
\begin{definition}\label{soliton1}
A complete Riemannian metric $g$ on a smooth manifold $M^n$ is called a \textbf{Ricci soliton} 
if and only if there exist a positive constant $\gamma$ and a vector field $X$ such that
\begin{equation}\label{soliton2}
	2\mathrm{Ric}(g)-2\gamma g+\mathcal{L}_Xg=0,
\end{equation}
where $\mathrm{Ric}(g)$ and $\mathcal{L}_X$ respectively denote the Ricci tensor of $g$ and the Lie derivative along $X$.
Moreover, if $X$ is the gradient vector field of a smooth function, $g$ is called a \textbf{gradient Ricci soliton}.
The Ricci soliton is said to be shrinking, steady and expanding according as $\gamma > 0$, $\gamma = 0$ and $\gamma <0$. 
\end{definition}
\noindent
If $X$ is zero, then $g$ is Einstein, in which case we say that $g$ is trivial. 
Due to Perelman (Remark 3.2, \cite{Perelman}), it is known that any Ricci soliton on a compact manifold is a gradient soliton. 
(See \cite{derdzinski} for a direct Riemannian proof of Perelman's result.)
It is also known that any nontrivial gradient Ricci soliton on a compact manifold is shrinking (\cite{Ham93}, \cite{Ivey}, see also \cite{Cao06})
with $n = \dim M \ge 4$. 
Examples of nontrivial compact K\"ahler-Ricci solitons have been constructed by Koiso \cite{Koiso90},
Cao \cite{Cao96} and Wang and Zhu \cite{Wang-Zhu}.
The main result of this paper is the following.

\begin{theorem}\label{main} Let $M^n$ be a compact smooth manifold with $n = \dim M \ge 4$. 
If $g$ is a non-trivial gradient shrinking Ricci soliton on $M$ satisfying (\ref{soliton2}), then
\begin{equation}\label{soliton3}
	d_g \ge \frac {10\pi}{13 \sqrt \gamma}
\end{equation}
where $d_g$ is the diameter of $M$ with respect to $g$.
\end{theorem}

We may compare this result with the case of Einstein metrics. Suppose in (\ref{soliton2}) $X = 0$ so that $g$ is an Einstein metric
satisfying $\mathrm{Ric}(g) = \gamma g$.
Then by Myers' theorem we have 
\begin{equation}\label{soliton4}
	d_g \le \sqrt{\frac{n-1}\gamma}\pi.
\end{equation}
Thus for compact Einstein manifolds the diameter is bounded from above while for nontrivial compact gradient shrinking Ricci solitons the diameter is bounded
from below by a universal constant.
Theorem \ref{main} also implies a gap result for gradient shrinking Ricci solitons on a compact Riemannian manifold, i.e., if $d_g$ is strictly less than $\frac{10\pi}{13\sqrt{\gamma}}$, then $g$ should be Einstein.
Other different types of gap theorems for gradient shrinking Ricci solitons are known.
For example, see \cite{Fer-Gar} and \cite{Lihaozhao08}.

The proof of Theorem \ref{main} is given by two steps. 
Assume that $g$ is a gradient shrinking soliton with respect to a gradient vector field $X$.
Let $f\in C^\infty(M)$ be the potential function of $X$.
Then, the equation (\ref{soliton2}) is equivalent to
\begin{equation*}
	R_{ij}-\gamma g_{ij}+\nabla_i\nabla_j f=0
\end{equation*}
where $R_{ij}$ denotes the Ricci curvature.
Since $\gamma$ is positive, $\mathrm{Ric}(g)+\mathrm{Hess}(f)$ is positive definite.
This means that Bakry-\'Emery geometry works on our case.
Let $\Delta_f$ be the corresponding Bakry-\'Emery Laplacian, which is defined by
$$
	\Delta_f=\Delta-\nabla f \cdot \nabla
$$
where $\Delta=g^{ij}\nabla_i\nabla_j$.
We normalize $f$ so that it satisfies
\begin{equation*}
	\int_M fe^{-f} dV_g=0.
\end{equation*}
The first step of the proof of Theorem \ref{main} is to show that $-2\gamma$ is an eigenvalue for $\Delta_f$. 
The second step is then to show that if $-\lambda$ is the first non-zero eigenvalue of $\Delta_f$ then $\lambda$ is bounded below by a universal constant:
\begin{equation}\label{soliton7}
		\lambda \ge \frac{\pi^2}{d_g^2}+\frac{31}{100}\gamma
\end{equation}
for $n\ge 3$. This estimate is an extension of a result of Ling \cite{ling07} in the case of ordinary Laplacian to the case of
Bakry-\'Emery Laplacian. A similar type of extension has been found by Lu and Rowlett \cite{LuRowlett} who extended a result of
Li and Yau \cite{LiYau80} in the case of ordinary Laplacian to the case of Bakry-\'Emery Laplacian. 

This paper is organized as follows. In section 2 we show that that $-2\gamma$ is an eigenvalue for $\Delta_f$. In section 3 we give a proof of
(\ref{soliton7}). In section 4 we prove Theorem \ref{main}.

\section{An eigenfunction for the twisted Laplacian}

Let $M$ be an $n$-dimensional compact Riemannian manifold and $g$ be a gradient shrinking soliton on $M$
so that we have 
\begin{equation}\label{soliton5}
	R_{ij}-\gamma g_{ij}+\nabla_i\nabla_j f=0
\end{equation}
with $\gamma$ positive, and thus $\mathrm{Ric}(g)+\mathrm{Hess}(f)$ is positive definite.
Let $\Delta_f$ be the corresponding Bakry-\'Emery Laplacian defined by
\begin{equation*}
	\Delta_f=\Delta-\nabla f \cdot \nabla
\end{equation*}
where $\Delta=g^{ij}\nabla_i\nabla_j$.
We normalize $f$ so that it satisfies
\begin{equation}\label{soliton6}
	\int_M fe^{-f} dV_g=0.
\end{equation}

\begin{lemma}\label{eigen}
The function $f$ is an eigenfunction of $\Delta_f$ with eigenvalue equal to $-2\gamma$.
\end{lemma}
\begin{proof}
Taking the covariant derivative of the left hand in (\ref{soliton5}) we have 
\begin{eqnarray*}
		0=
		\nabla_k(R_{ij}-\gamma g_{ij}+\nabla_i\nabla_j f)
	&=&
		\nabla_k R_{ij}+\nabla_i\nabla_k\nabla_j f
		- R_{ki}{}^l{}_j\nabla_l f
	\\
	&=&
		\nabla_k R_{ij}-\nabla_i R_{kj}
		- R_{ki}{}^l{}_j \nabla_l f.
\end{eqnarray*}
Taking the trace of the above on $i$ and $j$,
$$
	\nabla_k R - \nabla_i R_{k}{}^i - R_{k}{}^l\nabla_l f=0,
$$
where $R$ is the scalar curvature of $g$.
Then this and the contracted second Bianchi identity
$$
	\nabla_i R_k{}^i=\frac{1}{2}\nabla_k R
$$
induce
\begin{equation}\label{eq:ricci_scalar}
	\nabla_k R=2R_{k}{}^l\nabla_l f.
\end{equation}
From (\ref{soliton5}) and (\ref{eq:ricci_scalar}), we get
\begin{eqnarray*}
		\nabla_k (R-2\gamma f+|\nabla f|^2)
	&=&
		2R_k{}^l\nabla_l f-2\gamma\nabla_k f+2\nabla_k\nabla_l f\nabla^l f
	\\
	&=&
		2(R_{kl}-\gamma g_{kl}+\nabla_k\nabla_l f)\nabla^l f
	\\
	&=&
		0.
\end{eqnarray*}
Hence, there exists some constant $C$ such that
\begin{equation}\label{eq:constant}
	R-2\gamma f+|\nabla f|^2=C.
\end{equation}
Taking the trace of (\ref{soliton5}) on $i$ and $j$, we have
\begin{equation}\label{eq:trace_srs}
	R-n\gamma +\Delta f=0.
\end{equation}
From (\ref{eq:constant}) and (\ref{eq:trace_srs}), we have
\begin{equation}\label{eq:uptoconstant}
	\Delta_f f
	=\Delta f -|\nabla f|^2
	= -2\gamma f+C'
\end{equation}
where $C'$ is some constant.
Since
$$
	\int_M \Delta_f e^{-f} dV_g=0,
$$
the constant $C'$ in (\ref{eq:uptoconstant}) must be zero.
The proof of the lemma is completed.
\end{proof}

\begin{remark}\label{remark1} Lemma \ref{eigen} says that on a compact nontrivial gradient shrinking Ricci soliton, there always exists a 
eigenfunction of $\Delta_f$ with a fixed 
nonzero eigenvalue. 
A similar result holds for Fano manifolds, i.e. compact complex manifolds with positive first Chern class. If $M$ is such a manifold and
$g$ is a K\"ahler metric whose K\"ahler form represents $c_1(M)$ then there exists a smooth function $f$ such that 
$$R_{i\barj} = g_{i\barj} + \partial_i\barpartial_j f.$$
It is shown in \cite{futaki87} that if $M$ has non-zero holomorphic vector fields then $\Delta_f $ has eigenfunctions with eigenvalue equal to
$-1$. The gradient vector fields of the eigenfunctions are the holomorphic vector fields.
\end{remark}

\section{Eigenvalue estimates for Bakry-\'Emery Laplacian}
In this section, we extend Theorem 1 in \cite{ling07} to Bakry-\'Emery geometry.
We shall state the main result of this section in the end of this section (see Theorem \ref{thm:main_ling}).

Let $(M, g, \phi)$ be a Bakry-\'Emery manifold, that is to say, a triple of a Riemannian manifold $M$ with a Riemannian metric $g$ and a 
weighted volume form $e^{-\phi}dV_g$.
The Bakry-\'Emery Ricci curvature is defined by
$$
	\mathrm{Ric}_\phi:=\mathrm{Ric}(g)+\mathrm{Hess}(\phi),
$$
and the Bakry-\'Emery Laplacian is defined by
$$
	\Delta_\phi:=\Delta-\nabla\phi\cdot\nabla
$$
where $\mathrm{Ric}(g)$ denotes the Ricci tensor of $g$ and $\mathrm{Hess}(\phi)$ the Hessian of $\phi$.
Here we assume that $M$ is compact and has no boundary.
Also we assume that $\mathrm{Ric}_\phi$ is strictly positive as follows.
For some constant $K>0$,
\begin{equation}
	\mathrm{Ric}_\phi \ge (n-1)K\,g
\end{equation}
for all $x\in M$.

Let $u$ be an eigenfunction of the first non-zero eigenvalue $-\lambda$, i.e., $\Delta_\phi u=-\lambda u$.
We normalize $u$ as follows.
If $-\min u > \max u$, we replace $u$ by $-u$.
Otherwise, we keep $u$ unchanged.
Then, denoting $u/\max u$ by the same letter $u$ again, we can assume that $u$ satisfies
$$
	\max u=1,\,\, \min u=-k,\,\, 0<k\le 1.
$$
We define a function $v$ by
$$
	v:=[u-(1-k)/2][(1+k)/2].
$$
Then, $v$ satisfies
$$
	\max v=1,\,\, \min v=-1,
$$
and
$$
	\Delta_\phi v=-\lambda(v+a),
$$
where
$$
	a:=(1-k)/(1+k).
$$
Note that $0\le a < 1$.
%%%%%%
\begin{proposition}\label{prop:lemma3}
	$$
		\lambda \ge (n-1)K.
	$$
\end{proposition}
%%%%%%
\begin{remark}
This estimate is weaker than the one in \cite{ling07}.
\end{remark}
\begin{proof}
At any $x\in M$, we have
\begin{eqnarray*}
		\Delta(|\nabla v|^2)
	&=&
		\nabla^j\nabla_j(\nabla_i v \nabla^i v)
	\\
	&=&
		2\nabla^j(\nabla_j\nabla_i v \nabla^i v)
	\\
	&=&
		2\nabla^j\nabla_j\nabla_i v\nabla^i v+2|\nabla\nabla v|^2
	\\
	&=&
		2(
		\nabla_i(\Delta v)\nabla^i v
		+\mathrm{Ric}(\nabla v, \nabla v))
		+2|\nabla\nabla v|^2
	\\
	&\ge&
		2\{
		\nabla_i(-\lambda(v+a)
		+\nabla^j\phi\nabla_j v
		)\nabla^i v
		+(n-1)K|\nabla v|^2
	\\
	&&
		{-\nabla_i\nabla_j \phi\nabla^i v\nabla^j v}
		\}
		+2|\nabla\nabla v|^2
	\\
	&=&
		-2\lambda|\nabla v|^2
		{+2\nabla^i v\nabla_i\nabla_j v \nabla^j \phi}
		+2(n-1)K|\nabla v|^2+2|\nabla\nabla v|^2.
\end{eqnarray*}
We have
$$
	\nabla^j \phi\nabla_j(|\nabla v|^2)
	=
	2\nabla^i v \nabla_i \nabla_j v \nabla^j \phi.
$$
Then, we have
$$
	\frac{1}{2}\Delta_{\phi}(|\nabla v|^2)
	\ge
	((n-1)K-\lambda)|\nabla v|^2+|\nabla\nabla v|^2
	\ge
	((n-1)K-\lambda)|\nabla v|^2.
$$
Integrating both sides in the above inequality with respect to $e^{-\phi}dV_g$, we get the desired estimate.
The proof is completed.
\end{proof}
\begin{remark}
In \cite{ling07}, the integral of $|\nabla\nabla v|^2$ over $M$ with respect to $dV_g$ is estimated by the integral of $|\nabla v|^2$, then it implies that $\lambda \ge nK$.
However, the same way does not hold in the case of Bakry-\'Emery geometry.
In fact, the error terms, which come from the gap between $\Delta$ and $\Delta_\phi$, are not cancelled.
\end{remark}

%%%%%%
\begin{proposition}
	Let $b>1$ be an arbitrary constant. Then, $v$ satisfies
	\begin{equation}\label{eq:estimate_nabla_v}
		\frac{|\nabla v|^2}{b^2-v^2}\le \lambda (1+a).
	\end{equation}
\end{proposition}
%%%%%%
\begin{proof}
Consider the function on $M$ defined by
\begin{equation}
	P(x):=|\nabla v|^2+Av^2, \,\, A:=\lambda(1+a)+\varepsilon
\end{equation}
for small $\varepsilon >0$.
Let $x_0\in M$ be a point which $P$ attains its maximum.
The maximum principle implies  
\begin{equation}
	\nabla P(x_0)=0,\,\, \Delta P(x_0)\le 0.
\end{equation}
There are two cases, either $\nabla v(x_0)=0$ or $\nabla v(x_0)\neq 0$.

If $\nabla v(x_0)=0$, then
$$
	|\nabla v|^2+Av^2 \le P(x_0)=Av^2(x_0) \le A.
$$
So
\begin{eqnarray*}
		\frac{|\nabla v|^2}{b^2-v^2}
	\le
		\frac{1-v^2}{b^2-v^2}A
	\le
		A.
\end{eqnarray*}
Letting $\varepsilon \to 0$, we ge (\ref{eq:estimate_nabla_v}).

If $\nabla v(x_0)\neq 0$, then we rotate the local orthonormal frame about $x_0$ such that
$$
	|\nabla_1v(x_0)|=|\nabla(x_0)|\neq 0,\,\,\,
	\nabla_i v(x_0)=0 \,\, \mbox{for } i\ge 2.
$$
Since
$$
	\frac{1}{2}\nabla_1 P(x_0)
	=
	\nabla^1 v\nabla_1\nabla_1 v+Av\nabla_1 v
	=0,
$$
we have
$$
	\nabla_1\nabla_1 v=-Av.
$$
For $i\ge2$, since 
$$
	\frac{1}{2}\nabla_i P(x_0)
	=
	\nabla^j v\nabla_i\nabla_j v+Av\nabla_i v
	=
	\nabla^1 v\nabla_i\nabla_1 v
	=0,
$$
we have
$$
	\nabla_i\nabla_1 v=0.
$$
From these two equalities, we get
\begin{equation}\label{eq:nabla_nabla}
	|\nabla\nabla v|^2 \ge A^2v^2
\end{equation}
and
\begin{equation}\label{eq:nabla_nabla_nabla}
	\nabla^j v\nabla_j\nabla_i v=-Av\nabla_i v
\end{equation}
at $x_0\in M$.
Then, at $x_0\in M$, we have
\begin{eqnarray*}
		0
	&\ge&
		\frac{1}{2}\Delta P(x_0)
	=
		\frac{1}{2}\Delta(|\nabla v|^2+Av^2)
	\\
	&=&	
		|\nabla\nabla v|^2	
		+
		\nabla^j\nabla_j\nabla_i v \nabla^i v
		+
		A|\nabla v|^2
		+
		Av\Delta v
	\\
	&=&
		A^2v^2	
		+
		\nabla^j\nabla_j\nabla_i v \nabla^i v
		+
		A|\nabla v|^2
		+
		Av\Delta v
		\quad
		(\mbox{due to } (\ref{eq:nabla_nabla}))
	\\
	&=&
		A^2v^2	
		+
		\nabla^i v (\nabla_i(\Delta v))
		+
		\mathrm{Ric}(\nabla v,\nabla v)
		+
		A|\nabla v|^2
		+
		Av\Delta v
	\\
	&\ge&
		A^2v^2	
		+
		\nabla^i v (\nabla_i(-\lambda(v+a)
		{+\nabla_j\phi\nabla^j v}
		))
		+
		(n-1)K(\nabla v,\nabla v)
		\\
		&&
		{-\nabla_i\nabla_j \phi\nabla^i v\nabla^j v}
		+
		A|\nabla v|^2
		+
		Av(-\lambda(v+a)
		{+\nabla_i\phi\nabla^i v}
		)
	\\
	&=&
		Av^2-\lambda|\nabla v|^2+(n-1)K|\nabla v|^2
		+A|\nabla v|^2-A\lambda v^2-aA\lambda v
	\\
	&&
		(\mbox{due to }(\ref{eq:nabla_nabla_nabla}))
	\\
	&\ge&
		(A-\lambda)|\nabla v|^2+A(A-\lambda)v^2-aA\lambda v.
\end{eqnarray*}
Then, at $x_0\in M$, we have
\begin{eqnarray*}
		|\nabla v|^2+(A-\varepsilon)v^2
	&\le&
		-Av^2+\frac{aA\lambda v}{a\lambda +\varepsilon}	
		+(A-\varepsilon)v^2
	\\
	&=&
		-\varepsilon v^2+\frac{aA\lambda v}{a\lambda+\varepsilon}.
\end{eqnarray*}
Letting $\varepsilon\to 0$, we get
$$
	|\nabla v|^2(x)+Av^2(x) \le A
$$
for any $x\in M$.
Therefore, we can get the desired estimate as the former case.
The proof is completed.
\end{proof}
Define a function $Z$ by
$$
	Z(t):= \max_{x\in U(t)} \frac{|\nabla v|^2}{\lambda(b^2-v^2)}
$$
where $U(t):=\{x\in M \mid \sin^{-1}(v(x)/b)=t\}$.
Note that 
$$
	t\in [-\sin^{-1}(1/b), \sin^{-1}(1/b)].
$$
Let
$$
	c:=\frac{a}{b},\,\, 
	\alpha:=\frac{1}{2}(n-1)K, \,\,
	\delta:=\frac{\alpha}{\lambda}.
$$
%%%%%%%%%%%
\begin{proposition}\label{prop:theorem5}
	If the function $z:[-\sin^{-1}(1/b), \sin^{-1}(1/b)]\to \mathbb{R}$ satisfies the following
	\begin{enumerate}
		\item\label{item:1}
			$z(t)\ge Z(t)$ for all $t\in[-\sin^{-1}(1/b), \sin^{-1}(1/b)]$;
		\item\label{item:2}
			there exists some $x_0\in M$ such that $z(t_0)=Z(t_0)$ at $t_0=\sin^{-1}(v(x_0)/b)$;
		\item\label{item:3}
			$z(t_0)>0$;
\end{enumerate}
then we have
\begin{eqnarray}
	\nonumber
			0
	&\le&
			\frac{1}{2}\ddot{z}(t_0)\cos^2 t_0
			-
			\dot{z}(t_0)\cos t_0\sin t_0
			-
			z(t_0)
			+
			1
			+
			c\sin t_0
			-
			2\delta\cos^2 t_0
	\\
	\label{eq:test_estimate}
	&&
		-
		\frac{\dot{z}(t_0)}{4z(t_0)}\cos t_0
		\{
			\dot{z}(t_0)\cos t_0-2z(t_0)\sin t_0 +2\sin t_0 +2c
		\}.
\end{eqnarray}
\end{proposition}
%%%%%%%%%%%
\begin{proof}
Define
$$
	J(x):=\bigg\{
		\frac{|\nabla v|^2}{b^2-v^2}-\lambda z
	\bigg\}
	\cos^2 t
$$
where $t=\sin^{-1}(v(x)/b)$.
Since $z(t)\ge Z(t)$ and $z(t_0)=Z(t_0)$,
we have
$$
	J(x)\le 0 \,\,\, \mbox{for }\forall x\in M,
	\quad J(x_0)=0.
$$
Since $Z(t_0)>0$, we have
$$
	\nabla v(x_0)\neq 0.
$$
By the maximum principle, we have
$$
	\nabla J(x_0) =0,\,\,\,
	\Delta J(x_0)\le 0.
$$
Since
$
	\cos^2 t=1-(v/b)^2,
$
$J$ can be written by
$$
	J(x)=\frac{1}{b^2}|\nabla v|^2-\lambda z\cos^2t.
$$
Since $\nabla J(x_0)=0$, at $x_0\in M$
\begin{equation}\label{eq:nabla3_cos}
	\frac{2}{b^2}\nabla_j\nabla_i v\nabla^i v
	=
	\lambda\cos t
	\{
		\dot{z}\cos t-2z\sin t
	\}
	\nabla_j t.
\end{equation}
As before, we rotate the local orthonormal frame about  $x_0\in M$ such that
$$
	|\nabla_1v(x_0)|=|\nabla v(x_0)|
$$
and $\nabla_i v(x_0)=0$ for $i\ge 2$.
Note that $\nabla_i t(x_0)=0$ for $i\ge2$, because
\begin{equation}\label{eq:v_nabla_cos}
	\nabla_iv(x_0)=b\cos t_0\nabla_i t(x_0).
\end{equation}
From (\ref{eq:nabla3_cos}) and (\ref{eq:v_nabla_cos}), we have
\begin{equation}\label{eq:nabla3_cos-2}
	\nabla_1\nabla_1 v=\frac{\lambda b}{2}
	(\dot{z}\cos t-2z\sin t)
\end{equation}
and $\nabla_j\nabla_1 v=0\,\, (j\ge2)$ at $x_0\in M$.

Next we compute $\Delta J(x_0)$.
At $x_0\in M$, we have
\begin{eqnarray}
	\nonumber
		\Delta J(x_0)
	&=&
		\Delta\big(
			\frac{1}{b^2}|\nabla v|^2-\lambda z\cos^2 t
		\big)
	\\
	\nonumber
	&=&	
		\frac{2}{b^2}|\nabla\nabla v|^2
		+
		\frac{2}{b^2}\nabla^i v(\nabla^j\nabla_j\nabla_i v)
	\\
	\nonumber
	&&
		-\lambda\nabla^j
		(
			\dot{z}\cos^2 t\nabla_j t-2z\cos t\sin t\nabla_j t
		)
	\\
	\nonumber
	&=&	
		\frac{2}{b^2}|\nabla\nabla v|^2
	\\
	\nonumber
	&&	
		+
		\frac{2}{b^2}\nabla^i v(\nabla^j\nabla_j\nabla_i v)
	\\
	\nonumber
	&&
		-\lambda\ddot{z}\cos^2t|\nabla t|^2-\lambda\dot{z}\cos^2t(\Delta t)
	\\
	\label{eq:Delta_J}
	&&
		+4\lambda\dot{z}\cos t\sin t|\nabla t|^2
		-\lambda z\Delta(\cos^2t).
\end{eqnarray}
Before estimating the above, we prepare some calculations.
Remark that the following calculations are done at $x_0\in M$.
\begin{eqnarray}
	\label{eq:cal-1}
		|\nabla v|^2
	&=&
		\lambda	b^2z\cos^2 t;
		\quad
		(\mbox{due to } J(x_0)=0)
	\\
	\label{eq:cal-2}
		|\nabla t|^2
	&=&
		\frac{|\nabla v|^2}{b^2\cos^2 t}
		=
		\lambda z;
		\quad
		(\mbox{due to }(\ref{eq:cal-1}))
	\\
	\label{eq:cal-3}
		\frac{\Delta v}{b}
	&=&
		\Delta \sin t
		=
		\cos t(\Delta t)-\sin t|\nabla t|^2;
	\\
	\nonumber
		\Delta t
	&=&
		\frac{1}{\cos t}\big(
			\sin t|\nabla t|^2+\frac{\Delta v}{b}
		\big)
		\quad (\mbox{due to }(\ref{eq:cal-3}))
	\\
	\label{eq:cal-4}
	&=&
		\frac{1}{\cos t}
		\big(
			\lambda z\sin t 
			-
			\frac{\lambda}{b}(v+a)
			{+\frac{1}{b}
			(\nabla_i \phi \nabla^i v)
			}
		\big);
		\quad (\mbox{due to } (\ref{eq:cal-2}))
	\\
	\nonumber
		\Delta \cos^2 t
	&=&
		\Delta \big(1- \frac{v^2}{b^2} \big)
	\\
	\nonumber
	&=&
		-\frac{2}{b^2}|\nabla v|^2-\frac{2}{b^2}v\Delta v
	\\
	\label{eq:cal-5}
	&=&
		-2\lambda z \cos^2t-\frac{2}{b^2}v(-\lambda(v+a)
		{+\nabla_i\phi\nabla^i v}
		).
\end{eqnarray}

Now, at $x_0\in M$ we get
\begin{eqnarray}
	\nonumber
		\frac{2}{b^2}|\nabla\nabla v|^2
	&\ge&
		\frac{2}{b^2}(\nabla_1\nabla_1 v)^2
	\\
	\label{eq:cal-A}
	&=&
		\frac{\lambda^2}{2}(\dot{z})^2\cos^2t
		-
		2\lambda^2z\dot{z}\cos t\sin t +2\lambda^2z^2\sin^2t,
		\quad
		(\mbox{due to }(\ref{eq:nabla3_cos-2}))
\end{eqnarray}
and
\begin{eqnarray}
	\nonumber
		\frac{2}{b^2}\nabla^iv(\nabla^j\nabla_j\nabla_i v)
	&=&
		\frac{2}{b^2}\big(
			\nabla^jv\nabla_j(\Delta v)+\mbox{Ric}(\nabla v,\nabla v)
		\big)
	\\
	\nonumber
	&\ge&
		\frac{2}{b^2}\big(
			\nabla^j v\nabla_j(\Delta v)
			{-\nabla_i\nabla_j\phi\nabla^i v\nabla^j v}
			+
			2\alpha|\nabla v|^2
		\big)
	\\
	\nonumber
	&=&
		\frac{2}{b^2}\big(
			-\lambda |\nabla v|^2
			{+\nabla_i\phi\nabla^j v\nabla_j\nabla^i v}
			+2\alpha |\nabla v|^2
		\big)
	\\
	\label{eq:cal-B}
	&=&
		(-2\lambda^2+4\alpha\lambda)z\cos^2t
		{+\frac{2}{b^2}\nabla_i\phi\nabla^j v\nabla_j\nabla^i v},
	\\
	\nonumber
	&&
		(\mbox{due to }(\ref{eq:cal-1}))	
\end{eqnarray}
and
\begin{eqnarray}
	\nonumber
		-\lambda(\ddot{z}|\nabla t|^2+\dot{z}\Delta t)\cos^2t
	&=&
		-\lambda^2z\ddot{z}\cos^2t
		-
		\lambda\dot{z}\cos^2t
		\big(
			\frac{1}{\cos t}
		\big[
			\lambda z\sin t
	\\
	\nonumber
	&&
			-
			\frac{\lambda}{b}(v+a)
			+
			\frac{1}{b}
			(\nabla_i \phi\nabla^i v)
		\big]
		\big)
		\quad
		(\mbox{due to }(\ref{eq:cal-2}),\, (\ref{eq:cal-3}))
	\\
	\nonumber
	&=&
		-\lambda^2z\ddot{z}\cos^2t-\lambda^2z\dot{z}\cos t\sin t
		+\frac{1}{b}\lambda^2\dot{z}(v+a)\cos t
	\\
	\label{eq:cal-C}
	&&
		{-\frac{1}{b}\lambda \dot{z}\cos t
		\nabla_i\phi\nabla^i v},
\end{eqnarray}
and
\begin{eqnarray}
	\nonumber
		4\lambda\dot{z}\cos t\sin t|\nabla t|^2 - \lambda z \Delta \cos^2t
	&=&
		4\lambda^2z\dot{z}\cos t \sin t
		-
		\lambda z
		\big(
			-2\lambda z\cos^2t
	\\
	\nonumber
	&&
		+\frac{2}{b^2}\lambda v(v+a)
		-\frac{2}{b^2}
		v\nabla_i \phi \nabla^i v
		\big)
	\\
	\nonumber
	&&
		(\mbox{due to }(\ref{eq:cal-2}),\,\,(\ref{eq:cal-5}))
	\\
	\label{eq:cal-D}
	&=&
		4\lambda^2z\dot{z}\cos t\sin t +2\lambda^2z^2\cos^2t
	\\
	\nonumber
	&&
		-\frac{2}{b^2}\lambda^2z\sin t\cdot(v+a)
		{+\frac{2}{b^2}\lambda z v
		\nabla_i \phi\nabla^i v}.
\end{eqnarray}
Since (\ref{eq:nabla3_cos}) implies 
\begin{eqnarray*}
		\frac{2}{b^2}
		\nabla^j v \nabla_j\nabla^i v
	&=&
		\lambda(\cos t 
		\nabla^i t
		)[\dot{z}\cos t-2z\sin t]
	\\
	\nonumber
	&=&
		\frac{\lambda} {b}
		\nabla^i v
		[\dot{z}\cos t-2z\sin t]
	\\
	\nonumber
	&=&
		\frac{1}{b}\lambda\dot{z}\cos t 
		\nabla^i v
		-
		\frac{2}{b^2}\lambda z v 
		(\nabla^i v),
\end{eqnarray*}

getting $\Delta J(x_0)\le 0$, (\ref{eq:Delta_J}), (\ref{eq:cal-A}), (\ref{eq:cal-B}), (\ref{eq:cal-C}) and (\ref{eq:cal-D}) together, we have
\begin{eqnarray}
	\nonumber
		0
	&\ge&
		-\lambda^2z\ddot{z}\cos^2t
		+
		\frac{\lambda^2}{2}(\dot{z})^2\cos^2t
		+
		\lambda^2\dot{z}\cos t(z\sin t+c+\sin t)
	\\
	\nonumber
	&&
		+2\lambda^2z^2-2\lambda^2z-2\lambda^2cz\sin t+4\alpha\lambda z\cos^2t
\end{eqnarray}
at $x_0\in M$.
Dividing the two sides in the above inequality by $2\lambda^2 z(x_0)>0$, we have
\begin{eqnarray}
	\nonumber
		0
	&\ge&
		-\frac{1}{2}\ddot{z}(t_0)\cos^2t_0
		+\frac{1}{2}\dot{z}(t_0)\cos t_0
			\bigg(\sin t_0 +\frac{c+\sin t_0}{z(t_0)}\bigg)
		+z(t_0)
	\\
	\nonumber
	&&
		-1-c\sin t_0 +2\delta\cos^2t_0
		+\frac{1}{4z(t_0)}(\dot{z}(t_0))^2\cos^2t_0.
\end{eqnarray}
Therefore we get the desired inequality.
The proof is completed.
\end{proof}
Then, we have the same results as Corollary 6 and 7 in \cite{ling07} from Proposition \ref{prop:theorem5}.
\begin{corollary} Let the notations be as in Proposition \ref{prop:theorem5}.
\begin{enumerate}
\item\label{item:cor6}
If $z(t)$ satisfies the conditions (\ref{item:1}), (\ref{item:2}), (\ref{item:3}) in Proposition \ref{prop:theorem5}, 
$\dot{z}(t_0)\ge0$ and $1-c \le z(t_0) \le 1+a$, then we have
\begin{equation}\label{eq:cor6}
		0
		\le 
		\frac{1}{2}\ddot{z}(t_0)\cos^2t_0
		-
		\dot{z}(t_0)\cos t_0\sin t_0
		-
		z(t_0)
		+1+c\sin t_0-2\delta\cos^2t_0.
\end{equation}
\item\label{item:cor7}
Suppose that $a=0$. If $z(t)$ satisfies the conditions (\ref{item:1}), (\ref{item:2}), (\ref{item:3}) in Proposition \ref{prop:theorem5}, $\dot{z}(t_0)\sin t_0\ge0$ and $z(t_0) \le 1$, then we have
\begin{equation}\label{eq:cor7}
		0
		\le 
		\frac{1}{2}\ddot{z}(t_0)\cos^2t_0
		-
		\dot{z}(t_0)\cos t_0\sin t_0
		-
		z(t_0)
		+1-2\delta\cos^2t_0.
\end{equation}
\end{enumerate}
\end{corollary}
\begin{proof}
The proof is the same as \cite{ling07}, but we recall it for the reader's convenience.
It is sufficient to show that the last term in (\ref{eq:test_estimate}) is nonnegative.
First, let us prove the fist case (\ref{item:cor6}).
From the condition (\ref{item:3}) in Proposition \ref{prop:theorem5}, $\dot{z}(t_0)\ge0$, $\cos t_0\ge 0$ and $z(t_0)>0$, then it is sufficient to show that 
$$
	-z(t_0)\sin t_0+\sin t_0+c
$$
is nonnegative.
When $t_0\ge0$, we have
$$
	-z(t_0)\sin t_0+\sin t_0+c
	\ge
	-(1+a)\sin t_0 +\sin t_0 +c
	\ge
	a(\frac{1}{b}-\sin t_0)\ge0.
$$
The last inequality in above follows from that $|\sin t_0|=|\nu(t_0)/b|\le 1/b$.
When $t_0 <0$, we have
$$
	-z(t_0)\sin t_0+\sin t_0+c
	\ge
	-(1-c)\sin t_0 +\sin t_0 +c
	\ge
	c(1+\sin t_0)\ge0.
$$
Next we shall prove the second case (\ref{item:cor7}).
In this case, it is sufficient to show that
$$
	\dot{z}(t_0)(-z(t_0)\sin t_0+\sin t_0)=(1-z(t_0))\dot{z}(t_0)\sin t_0
$$
is nonnegative, because $c=0$.
It follows from the assumptions.
Therefore, the proof is completed.
\end{proof}
%%%%%%%
Moreover, we also have the following results which are same as Theorem 8 and Theorem 9 in \cite{ling07}.
%%%%%%%
\begin{proposition}\label{prop:theorem8}
	Assume that $a> 0$ and $\mu\delta \le \frac{4}{\pi^2}a$ for a constant $\mu\in(0,1]$.
	Then, we have
	$$
		\lambda\ge
		\frac{\pi^2}{d^2}+\frac{\mu}{2}(n-1)K
		=\frac{\pi^2}{d^2}+\mu\alpha.
	$$
	Here $d$ denotes the diameter of $(M,g)$.
\end{proposition}
\begin{proposition}\label{prop:theorem9}
	Assume that $a=0$.
	Then, we have
	$$
		\lambda \ge \frac{\pi^2}{d^2}+\frac{1}{2}(n-1)K.
	$$
\end{proposition}
%%%%%%%
\begin{proof}
The proofs of Propositions \ref{prop:theorem8} and \ref{prop:theorem9} are the same as \cite{ling07}.
So, we shall give just the outline of proofs for the reader's convenience (See the original paper \cite{ling07} for the full details).
The key point of the proofs is to choose right functions $z(t)$ in (\ref{eq:cor6}) and (\ref{eq:cor7}) enough to prove.

First, let us see the proof of Proposition \ref{prop:theorem8}.
Let $\mu_\epsilon:=\mu-\epsilon>0$ for a small positive constant $\epsilon$.
Take $b>1$ close to $1$ such that $\mu_\epsilon <\frac{4}{\pi^2}c$.
Let
$$
	z(t):=1+c\eta(t)+\mu_\epsilon\delta\xi(t),
$$
where
$$
	\xi(t):=
	\frac
	{\cos t^2+2t\sin t\cos t +t^2-\frac{\pi^2}{4}}
	{\cos t^2}
	\quad
	\mbox{for }
	t\in \bigl[-\frac{\pi}{2}, \frac{\pi}{2}\bigr],
$$
and
$$
	\eta(t):=
	\frac
	{\frac{4}{\pi}t+\frac{4}{\pi}\cos t\sin t -2\sin t}
	{\cos^2 t}
	\quad
	\mbox{for }
	t\in \bigl[ -\frac{\pi}{2}, \frac{\pi}{2} \bigr].
$$
The needed properties of $\xi(t)$ and $\eta(t)$ for the proofs are studied in Section 4 of \cite{ling07}.
By using such properties and (\ref{eq:cor6}), we can show
\[
	Z(t)\le z(t),
	\quad
	\mbox{for }
	t\in [-\sin^{-1}(1/b), \sin^{-1}(1/b)].
\]
The above implies 
\begin{equation}\label{eq:lambda}
	\sqrt{\lambda} \ge \frac{|\nabla t|}{\sqrt{z(t)}}
	\quad
	\mbox{for }
	t\in [-\sin^{-1}(1/b), \sin^{-1}(1/b)].
\end{equation}
Let $q_1$ and $q_2$ be the two points in $M$ such that $\nu(q_1)=-1$ and $\nu(q_2)=1$ respectively.
Let $L$ be the minimum geodesic between $q_1$ and $q_2$.
Integrating the both sides of (\ref{eq:lambda}) along $L$ and changing variable, then we have
$$
	\sqrt{\lambda}d
	\ge
	\int_L \frac{|\nabla t|}{\sqrt{z(t)}}dl
	=
	\int^{\frac{\pi}{2}}_{-\frac{\pi}{2}}\frac{1}{\sqrt{z(t)}}dt
	\ge
	\frac{\big(\int^{\pi/2}_{-\pi/2}dt\big)^{3/2}}{\big(\int^{\pi/2}_{-\pi/2}z(t)dt\big)^{1/2}}
	\ge
	\biggl(
		\frac{\pi^3}{\int_{-\pi/2}^{\pi/2}z(t)dt}
	\biggr)^{1/2}.
$$
From the properties of $\xi(t)$ and $\eta(t)$, and the definition of $z(t)$, we can show that 
$$
	\int_{-\pi/2}^{\pi/2}z(t)dt=(1-\mu_\epsilon\delta)\pi.
$$
Hence we have
\begin{equation}\label{eq:lambda2}
	\lambda \ge 
	\frac{\pi^2}{(1-\mu_\epsilon\delta)d^2}.
\end{equation}
From $\mu_\epsilon<\mu\le1$ and 
$$
	\delta=\frac{\alpha}{\lambda}\le \frac{(n-1)K}{2}\cdot \frac{1}{(n-1)K}=\frac{1}{2}
	\quad
	(\mbox{due to Proposition \ref{prop:lemma3}} ),
$$
we have
$
	1-\mu_\epsilon\delta\ge\frac{1}{2}>0.
$
Then (\ref{eq:lambda2}) implies
$$
	\lambda 
	\ge 
	\frac{\pi^2}{d^2}+\lambda\mu_\epsilon\delta
	=
	\frac{\pi^2}{d^2}+\mu_\epsilon\alpha.
$$
Letting $\epsilon\to 0$, we have the desired inequality.

Next, let us see the proof of Proposition \ref{prop:theorem9}.
In this case, the test function is defined by
$$
	y(t)=1+\delta\xi(t).
$$
Then, by using (\ref{eq:cor7}), we can get $Z(t)\le y(t)$ for $t\in [-\pi/2, \pi/2]$.
The rest of the proof is similar to that of Proposition \ref{prop:theorem8}.
\end{proof}
%%%%
\begin{theorem}\label{thm:main_ling}
	Let $(M,g, \phi)$ be an $n$-dimensional compact Bakry-\'Emery manifold without boundary.
	Assume that there exists a constant $K>0$ such that
	$$
		\mathrm{Ric}_\phi \ge (n-1)K\,g
	$$
	for all $x\in M$.
	Then the first non-zero eigenvalue $\lambda$ of the Bakry-\'Emery Laplacian $\Delta_\phi$ on $M$ satisfies
	$$
		\lambda \ge \frac{\pi^2}{d^2}+\frac{31}{100}(n-1)K
		\quad
		\mbox{for } n\ge 2,
	$$
	where $d$ denotes the diameter of $(M,g)$.
\end{theorem}
%%%%%%%%
\begin{proof}
The proof is almost same as in \cite{ling07}, because all necessary results for the proof hold  in the same way as in \cite{ling07} except Proposition \ref{prop:lemma3}.
The proof in \cite{ling07} are proceeded separately in the following cases;
\begin{itemize}
	\item 
		Case (A): $a=0$.
	\item
		If $0<a<1$,
			\begin{itemize}
				\item 
					Case (B-1): $\frac{\pi^2}{4}\delta \le a$.
				\item
					If $a<\frac{\pi^2}{4}\delta$,
						\begin{itemize}
							\item
								Case (B-2-a): $0.765 \le a $.
							\item
								If $0< a < 0.765$,
									\begin{itemize}
										\item
											Case (B-2-b1): $1.53\delta\le a $.
										\item
											Case (B-2-b2): $a <1.53\delta$.
									\end{itemize}
						\end{itemize}
			\end{itemize}
\end{itemize}
The case whose proof requires Proposition \ref{prop:lemma3} is only Case (B-a-2): $0.765 \le a <\frac{\pi^2}{4}\delta$.
This means that we can prove the other cases in the same way and get the same statement as in \cite{ling07}.
For the reader's convenience, we shall give the outline of the proofs except Case (B-a-2).
In the cases except (B-a-2) and (B-2-b2), we can prove by applying Propositions \ref{prop:theorem8} and \ref{prop:theorem9} directly.
More precisely, we can apply
\begin{center}
	\begin{tabular}{lll}
			Proposition \ref{prop:theorem9} &  & to Case (A);
		\\
			Proposition \ref{prop:theorem8} & for $\mu=1$ & to Case (B-1);
		\\
			Proposition \ref{prop:theorem8} & for $\mu=\frac{4}{\pi^2}\frac{a}{\delta}$ 
			& to Case (B-2-b1).
	\end{tabular}
\end{center}
Then, from Cases (A), (B-1) we can get
$$
	\lambda \ge \frac{\pi^2}{d^2}+\frac{1}{2}(n-1)K.
$$
From Case (B-2-b1) we can get
\begin{equation}\label{eq:31/100}
	\lambda \ge \frac{\pi^2}{d^2}+\frac{31}{100}(n-1)K.
\end{equation}
The proof in Case (B-2-b2) is essentially similar to that of Proposition \ref{prop:theorem8}, although it is more technical than other cases.
The essential part of the proof is to show $Z(t)\le z(t)$ on $[-\sin^{-1}(1/b),\sin^{-1}(1/b)]$ with respect to a new test function
$$
	z(t):=1+c\eta(t)+(\delta-\sigma c^2)\xi(t)
	\quad
	\mbox{for }
	t\in [-\sin^{-1}(1/b),\sin^{-1}(1/b)],
$$
where
$$
	\sigma
	=
	\frac{\tau}
	{\biggl(
		\bigl[
			\frac{3}{2}-\frac{\pi^2}{8}
			-\bigl(\frac{\pi^2}{32}-\frac{1}{6}\bigr)\frac{153}{100}
		\bigr]
		\frac{200}{153}
		-\frac
		{\bigl(\frac{8}{3\pi}-\frac{\pi}{4}\bigr)^2}
		{\bigl[-1+(12-\pi^2)\frac{100}{153}\bigr]}
	\biggr)c}
	>0.
$$
Then, we can get (\ref{eq:31/100}).

Next, let us consider Case (B-a-2).
Since $\big(\frac{4}{\pi^2}\frac{a}{\delta}\big) \in (0,1]$ and $\big(\frac{4}{\pi^2}\frac{a}{\delta}\big)\delta =\frac{4}{\pi^2}a$, we can apply Proposition \ref{prop:theorem8} by putting $\mu=\big(\frac{4}{\pi^2}\frac{a}{\delta}\big)$.
Then we have
$$
	\lambda 
	\ge 
	\frac{\pi^2}{d^2}
	+
	\big(\frac{4}{\pi^2}\frac{a}{\delta}\big)\alpha
	=
	\frac{\pi^2}{d^2}
	+\frac{4a}{\pi^2}\lambda.
$$
On the other hand, Proposition \ref{prop:lemma3} implies
$$
	\lambda \ge 
	{(n-1)}
	K
	=
	{2\alpha}.
$$
%%%%%%%%
\begin{remark}
This inequality corresponds to
$
	\lambda \ge \frac{2n}{n-1}\alpha
$
in pp.397,  \cite{ling07}.
\end{remark}
%%%%%%%%
From these two inequalities, we have
\begin{eqnarray*}
		\lambda
	&\ge&
		\frac{\pi^2}{d^2}+
		{\frac{8a}{\pi^2}\alpha}
	\\
	&\ge&
		\frac{\pi^2}{d^2}+
		{\frac{8(0.765)}{\pi^2}\alpha}.
\end{eqnarray*}
Since
$$
	\frac{8(0.765)}{\pi^2}>0.62=\frac{31}{50},
$$
we get
$$
	\lambda
	\ge
	\frac{\pi^2}{d^2}+
	\frac{31}{50}\alpha
	=
	\frac{\pi^2}{d^2}+
	\frac{31}{100}(n-1)K,
$$
which is (\ref{eq:31/100}).
\end{proof}
%%%%%%
%%%
\begin{remark}
In \cite{ling07}, it is proved that we can get the better result when $n=2$ than $n\ge3$ for ordinary Laplacian.
However, since our Proposition \ref{prop:lemma3} is weaker than Lemma 3 in \cite{ling07}, the same argument does not hold for Bakry-\'Emery Laplacian.
\end{remark}
%%%

\begin{corollary}\label{Cor}
Let $(M,g)$ be a compact nontrivial gradient shrinking soliton with $\dim M \ge 4$ satisfying
\begin{equation*}
	R_{ij}-\gamma g_{ij}+\nabla_i\nabla_j f=0.
\end{equation*}
Let $\Delta_f$ be the corresponding Bakry-\'Emery Laplacian defined by
$$
	\Delta_f=\Delta-\nabla f \cdot \nabla
$$
where $\Delta=g^{ij}\nabla_i\nabla_j$.
If $-\lambda$ is the first non-zero eigenvalue of $\Delta_f$ then we have
\begin{equation*}
		\lambda \ge \frac{\pi^2}{d^2}+\frac{31}{100}\gamma.
\end{equation*}
\end{corollary} 
\begin{proof}
We have only to apply Theorem \ref{thm:main_ling} with $\gamma = (n-1)K$.
\end{proof}

\section{Proof of Theorem \ref{main}}

\begin{proof}[Proof of Theorem \ref{main}]
Let $(M,g)$ be a compact nontrivial gradient shrinking soliton with $\dim M \ge 4$ satisfying
\begin{equation*}
	R_{ij}-\gamma g_{ij}+\nabla_i\nabla_j f=0.
\end{equation*}
Let $\Delta_f$ be the corresponding Bakry-\'Emery Laplacian defined by
$$
	\Delta_f=\Delta-\nabla f \cdot \nabla
$$
where $\Delta=g^{ij}\nabla_i\nabla_j$. By Lemma \ref{eigen} $\Delta_f$ has an eigenvalue $-2\gamma$.
Thus  by Corollary \ref{Cor} we have 
 \begin{equation}
	2\gamma \ge \frac{\pi^2}{d^2}+\frac{31}{100}\gamma,
\end{equation}
from which (\ref{soliton3}) follows easily. This completes the proof of Theorem \ref{main}.
\end{proof}


\begin{thebibliography}{99}

\bibitem{Cao96}H.-D.~Cao, Existence of gradient K\"ahler-Ricci solitons, Elliptic and Parabolic Methods in Geometry
(Minneapolis, MN, 1994), A. K. Peters (ed.), Wellesley, MA, 1996, 1-16.

\bibitem{Cao06}H.-D.~Cao, Geometry of Ricci solitons. Chinese Ann. Math. Ser. B 27 (2006), no. 2, 121--142. 


\bibitem{derdzinski}A.~Derdzinski, Compact Ricci solitons, preprint.


\bibitem{Fer-Gar}M.~Fern\'andez-L\'opez and E.~Garcia-Rio, Diameter bounds and gap theorems for compact Ricci solitons, preprint.


\bibitem{futaki87}A.~Futaki, The Ricci curvature of symplectic quotients of Fano manifolds. Tohoku Math. J. (2) 39 (1987), no. 3, 329--339. 


\bibitem{Ham88} R.S.~Hamilton, The Ricci flow on surfaces. Mathematics and general relativity (Santa Cruz, CA, 1986), 237--262, 
Contemp. Math., 71, Amer. Math. Soc., Providence, RI, 1988. 

\bibitem{Ham93}R.S.~ Hamilton,  The formation of singularities in the Ricci flow, Surveys in Differential Geometry 
(Cambridge, MA, 1993), 2, International Press, Combridge, MA, 1995, 7-136.



\bibitem{Ivey} T.~Ivey, Ricci solitons on compact three-manifolds. Differential Geom. Appl. 3 (1993), no. 4, 301--307. 

\bibitem{Koiso90}N.~Koiso, On rotationally symmetric Hamilton's equation for K\"ahler-Einstein metrics. 
Recent topics in differential and analytic geometry, 327--337, 
Adv. Stud. Pure Math., 18-{\rm I}, Academic Press, Boston, MA, 1990. 

\bibitem{Lihaozhao08} H.~Li, Gap theorems for K\"ahler-Ricci solitons.
 Arch. Math. 91  (2008),  no. 2, 187--192.

\bibitem{LiYau80}P.~Li and S.-T.~Yau, Estimates of eigenvalues of a compact Riemannian manifold. 
Geometry of the Laplace operator (Proc. Sympos. Pure Math., Univ. Hawaii, Honolulu, Hawaii, 1979), 
pp. 205--239, Proc. Sympos. Pure Math., XXXVI, Amer. Math. Soc., Providence, R.I., 1980.

\bibitem{ling07} J.~Ling, 
Lower bounds of the eigenvalues of compact manifolds with positive Ricci curvature, Ann. Global Anal. Geom. 31 (2007), no. 4, 385--408. 

\bibitem{LuRowlett}Z.~Lu and J.~Rowlette : The fundamental gap, preprint, arXiv:1003.0191.

\bibitem{Perelman}G.~Perelman, The entropy formula for the Ricci flow and its geometric applications, 
http://arXiv.org/abs/math.DG/02111159, 2002.

\bibitem{Wang-Zhu} X.-J. Wang and X. Zhu : 
K\"ahler-Ricci solitons on toric manifolds with positive first Chern class,
Adv. Math.  188  (2004),  no. 1, 87--103.

\end{thebibliography}
\end{document}